\documentclass{amsart}
\usepackage{amssymb,amsmath,amsfonts,pictex,graphicx,fullpage,etex,color,tikz}
\usetikzlibrary{shapes,shapes.arrows,snakes,positioning}

\newcommand{\E}{{\sf e}}

\newcommand{\V}{{\sf v}}

\newcommand{\W}{{\sf w}}
\newcommand{\X}{{\sf x}}
\newcommand{\Y}{{\sf y}}

\newcommand{\Z}{\mathbb{Z}}
\newcommand{\R}{\mathbb{R}}

\newcommand{\N}{\mathbb{N}}

\newcommand{\std}{{\sf std}}
\newcommand{\muL}{\mu\raisebox{-3pt}{\scriptsize{\ensuremath L}}}
\newcommand{\normL}[1]{\| #1 \| \raisebox{-3pt}{\scriptsize{\ensuremath L}}}
\newcommand{\muTL}{\mu\raisebox{-3pt}{\scriptsize{\ensuremath TL}}}
\newcommand{\normTL}[1]{\| #1 \| \raisebox{-3pt}{\scriptsize{\ensuremath TL}}}

\definecolor{thingie}{rgb}{.75,.4,.75}
\definecolor{brown}{rgb}{.545,.27,.075}

\DeclareMathOperator{\Cube}{Cube}
\DeclareMathOperator{\Orth}{Orth}
\DeclareMathOperator{\Sphere}{Sphere}

\DeclareMathOperator{\vol}{Vol}

\DeclareMathOperator{\AD}{\hat{E}}

\newcommand{\stimes}{\mspace{1mu}}

\theoremstyle{plain}
\newtheorem{theorem}{Theorem}
\newtheorem{Thm}[theorem]{Theorem}
\newtheorem{proposition}[theorem]{Proposition}
\newtheorem{prop}[theorem]{Proposition}

\newtheorem{lemma}[theorem]{Lemma}
\newtheorem{corollary}[theorem]{Corollary}

\newtheorem{conjecture}[theorem]{Conjecture}

\newtheorem*{sphere-asymp}{Theorem \ref{sphere-asymp}}
\newtheorem*{ball-asymp}{Corollary \ref{ball-asymp}}

\theoremstyle{definition}
\newtheorem{remark}[theorem]{Remark}

\newtheorem{definition}[theorem]{Definition}

\begin{document}

\title{Statistical hyperbolicity in groups}
\author[Duchin Leli\`evre Mooney]{Moon Duchin, Samuel Leli\`evre, and Christopher Mooney}
\date{\today}

\begin{abstract}
In this paper, we introduce a geometric
statistic called the {\em sprawl} of a group with respect to a generating set, based on
the average distance in the word metric between pairs of words of equal length.  The
sprawl  quantifies a certain obstruction to hyperbolicity.  Group presentations with maximum
sprawl (i.e., without this obstruction) are called {\em statistically hyperbolic}.
We first relate sprawl to curvature and show that nonelementary hyperbolic groups are statistically hyperbolic,
then give some results for products and for certain solvable groups.
In free abelian groups, the word metrics asymptotically approach norms induced by convex polytopes, 
causing the study of sprawl to 
reduce to a problem in convex geometry.  
We  present an algorithm that computes sprawl exactly for any generating set, thus quantifying the 
failure of various presentations of $\Z^d$ to be hyperbolic. 
This leads to a conjecture about the extreme values, with a connection to the classic Mahler conjecture.
\end{abstract}

\maketitle

\section{Introduction}

We will define and study  a new geometric 
statistic for groups in this paper, called the {\em sprawl} of a group (with respect to a
generating set).  Sprawl measures the average distance between pairs of points on the
spheres in the word metric, normalized by the radius, as the spheres get large.  
This gives a numerical measure of the asymptotic shape of spheres that can be studied for arbitrary 
finitely generated groups and locally finite graphs.

To be precise, let
$$E(G,S) := \lim_{n\to\infty}\frac{1}{|S_n|^2}\sum_{x,y\in S_n} \frac 1n \ d(x,y),$$ 
provided this limit exists.  Note that since $0\le d(x,y)\le 2n$, the value is always
between 0 and 2.  By way of interpretation, note that $E=2$ means that one can almost
always pass through the origin when traveling between any two points on the sphere without
taking a significant detour.  (The name is intended to invoke urban sprawl:  a higher value
means a lack of significant shortcuts between points on the periphery of the ``city.")

As we will see, this statistic is not quasi-isometry invariant but nonetheless captures
interesting features of the large-scale geometry, to be developed in \S\ref{hyp}.
Sprawl has connections to other 
geometric statistics such as divergence, almost-convexity, and discrete Ricci curvature.
After explaining why this statistic detects curvature properties, 
we  show below that non-elementary
hyperbolic groups always have $E(G,S)=2$ for any generating set, so we can think of $2-E$
as quantifying an obstruction to hyperbolicity in groups.  We give some results about sprawl for 
non-hyperbolic groups, including product groups and some solvable examples (lamplighter groups).

\subsection*{Free abelian groups and convex geometry}

For free abelian groups, there are particularly clear results on the asymptotic shape of spheres and the 
distribution of their points that allow us to compute the sprawl.  
As we will review below in \S\ref{limit-measure}, the word metrics on $\Z^d$
are close at large scale to certain norms, and the points of the spheres are distributed in a way that tends 
to a limit measure on the unit sphere in the norm.  
This allows us to replace an asymptotic computation on large discrete spheres 
by a finite computation: integrating average distance on a polytope against an appropriate measure.
We give an algorithm for performing this calculation for arbitrary $(\Z^d,S)$ in \S\ref{cutline}.
Though we can compute sprawl exactly for any finite presentation of $\Z^d$,  it is still an interesting problem 
to find the extremal values over all generating sets.   That is, we are studying a group statistic that depends on 
the choice of generators, but how much can it vary?  
This becomes a (possibly hard) problem in convex geometry, which we will study below.  

\begin{definition}
A {\em convex body} is a convex set in $\R^d$ with interior.    
A {\em perimeter} is the boundary of a centrally symmetric convex body in $\R^d$.
\end{definition}
\noindent (To emphasize this point: 
we are using the word ``perimeter" in a special way, which includes the assumption of 
central symmetry.  Accordingly, we assume that our generating sets $S$ for $\Z^d$ are symmetric, so that $S=-S$.)

We will show that a generating set $S$ for $\Z^d$ induces a perimeter $L$ in a very simple way ($L$ is just 
the boundary of the convex hull of $S$ in $\R^d$) and that the sprawl $E(\Z^d,S)=E(L)$ depends only on 
$L$.  Furthermore $E(L)=E(TL)$ for linear transformations $T$, so
sprawl gives an affine geometric invariant:  average distance between two points on the perimeter,
where both the distance and the measure have natural intrinsic definitions with respect to the shape.
We conjecture that the cube and the sphere are the extreme shapes in every dimension, and we give
some rigorous evidence for that in \S\ref{plane}-\ref{high-rank}.
This would mean, for example, that over all generating sets for $\Z^2$, the values achieved by sprawl 
are pinched between $4/\pi \approx 1.273$ and $4/3\approx 1.333$.    
This extremization problem 
resembles the well-known Mahler conjecture in convex geometry, a parallel  developed
in the last section below.

\subsection*{Acknowledgments}

We thank Alex Eskin, Ralf Spatzier, and Greg Kuperberg.
The first author is 
partially supported by NSF grant DMS-0906086, 
the second author is partially supported by ANR grants 06-BLAN-0038
and Project Modunombres, and 
 the third author is partially 
supported by NSF grant RTG-0602191.

%%%%%%%%%%%%
\section{Hyperbolic groups and statistically hyperbolic groups}\label{hyp}

In a graph, let us adopt the convention that for a real number $r\ge 0$, the
notation $S_r$ denotes the metric sphere of radius $\lfloor r \rfloor$.  We will study the
Cayley graph as the metric model of a group, adopting the convention that the points of
our metric space are the vertices (that is, the elements of the group), endowed with the
distance induced by the edges (the word metric).  We will write
$\beta(r):=\#B_r(e)$ to denote the number of group elements in the closed ball of radius
$r$ about the identity (or by translation-invariance, about any other center) in the
group.

\subsection{Hyperbolicity}

A metric space is called {\em $\delta$-hyperbolic} (or just {\em hyperbolic}, without
specifying a value $\delta$) if every geodesic triangle has the property that each side is
contained in the $\delta$-neighborhood of the union of the other two sides.  In such a
space, suppose two geodesic rays share a common endpoint.  Then if they become separated
by $2\delta$ at time $t_0$, they must subsequently diverge completely: the two subrays
after this separation can be concatenated to form a complete quasigeodesic, because for
$t>t_0$, any geodesic segment connecting $\gamma_1(t)$ and $\gamma_2(t)$ must return to a
$2\delta$-neighborhood of $\gamma_i(t_0)$.  This means that the distance between
$\gamma_1(t)$ and $\gamma_2(t)$ is at least $2(t-t_0-\delta)$.  Since we have strong
estimates on the distance after the rays stop fellow-traveling, our task for hyperbolic
groups will be to get quantitative control of the fellow-traveling.

To illustrate the issues involved in finding the sprawl of a group, 
first consider the free (nonabelian) group $F_2$ with its
standard generating set.  (Here and from now on, $\std$ will denote the standard
generating set for a group).  The Cayley graph is a $4$-regular tree, and to evaluate
the average on the sphere directly, one forms a finite sum by fixing one point on the
sphere and then counting the other points of $S_n$ at various distances from the first:
$$\sum_{y\in S_n}d(x,y) = \frac 34 (2n) + \frac 14 \frac 23 (2n-2) + \frac 14 \frac 13 \frac 23 (2n-4) +
\frac 14 \frac 13 \frac 13 \frac 23 (2n-6) + \cdots  + 0.$$
As $n\to\infty$, this can be evaluated using a geometric sum, and one computes in this way
that $E(F_2,\std)=2$.  This argument, however, is sensitive to the choice of generating
set.  What would happen for some other generating set?  Does the $\delta$-hyperbolicity of the model 
space suffice?  The answer is ``No'' in general.
One can easily construct trees with sprawl any number between 0 and 2, trees where sprawl does
not exist, and trees where sprawl depends on basepoint.  These trees are highly nonhomogeneous
and are not quasi-isometric to any group.

\begin{remark}[Sprawl and classical curvature]
Moving beyond locally finite graphs and groups, we can define the sprawl for metric
spaces that have natural measures on spheres. Instead of counting measure one may take Hausdorff measure
in the appropriate dimension, for example.  Thus for a space and measure $(X,\mu)$, 
we can write 
$$E(X):= \lim_{r\to\infty} \frac{1}{\mu(S_r)^2} \int_{S_r \times S_r} \frac 1r d(x,y) \  d\mu^2.$$

One can quickly show that the hyperbolic plane (and 
thus hyperbolic space of any dimension) has $E=2$:  for two rays making angle $\theta$ at
their common basepoint,
$d(\gamma_1(t),\gamma_2(t))\ge 2t - c(\theta)$, where $c(\theta)$ is a constant depending on $\theta$.

Indeed, it is not hard to identify a relationship between sprawl and curvature:  if
$E_r$ is defined to be the average distance between pairs of points on $S_r$
and $M_\kappa$ is the model space of constant sectional curvature
$\kappa$, it is easily observed 
that for every fixed value of  $r$, the values $E_r(M_\kappa)$ are strictly decreasing in $\kappa$ 
(taking $\kappa\le \pi^2r^2$ so that $S_r$ is non-empty).

However, a $\delta$-hyperbolic space, indeed even a
tree, need not have $E=2$, and exponential growth of balls or spheres does not suffice.
For instance, consider modifying the four-regular tree by choosing one axis and modifying the degree at
each vertex in that axis as a function of distance from the origin.  
Examples constructed in this way can achieve all values $0\le E\le 2$, and can also have
$E$ not exist or depend on basepoint.  Thus to prove that hyperbolic groups have maximal
sprawl, it is essential to make use of the homogeneity guaranteed by a transitive group
action.  We will use this by appealing to a strong result of Michel Coornaert giving definite exponential growth 
(not just a growth rate but furthermore a bound on the coefficients) for hyperbolic groups.
\end{remark}

\begin{remark}[Divergence, almost-convexity, discrete Ricci curvature]
Recall that sprawl is measured by
computing the distances between pairs of points $x,y\in S_n$, then taking the average and letting $n\to\infty$.
At least three other geometric statistics also study the geometry of pairs of points in the sphere.
\begin{itemize}
\item {\em Divergence} is measured by 
%finding the maximal distance between pairs of points 
%on large metric spheres $S_n$, where distance is measured in the induced length metric on $X\setminus B_n$.  
minimizing the length of paths between $x,y\in S_n$ such that the path lies outside of $B_n$, 
then taking the sup and letting $n\to\infty$.
This is widely studied for groups, for instance in \cite{gersten1,gersten2,kap-leeb,dr}.  
\item {\em Almost-convexity} for groups is measured by 
%finding geodesics between pairs of points on large 
%metric spheres $S_n$, then measuring how far outside of $B_n$ they may travel.  
minimizing the length of paths between $x,y\in S_n$ such that the path lies inside of $B_n$.
This was defined by Cannon in \cite{cannon} and further explored in many papers, such as 
\cite{her-mei,eld-her,cle-tab}.
\item Ricci curvature for manifolds is defined by considering infinitesimal spheres at a pair of basepoints, 
and measuring the average distances between corresponding points on the spheres.  If that distance is greater
than the distance between basepoints, then the curvature is negative; if smaller, then the curvature is positive; and
if equal, then the curvature is zero.  {\em Discrete Ricci curvature} mimics this construction in a manner usable 
for groups by measuring distances between corresponding points in metric spheres at different basepoints.
This was defined by Yann Ollivier in \cite{olliv1} and compared to optimal transport
definitions of C\'edric Villani and coauthors in \cite{oll-vil}.
\end{itemize}
Thus the definition of sprawl gives it a family resemblance to other synthetic curvature 
conditions that have already proved useful.  
\end{remark}

Recall that a hyperbolic group is called {\em elementary} if it is finite or has a
finite-index cyclic subgroup.  

\begin{Thm}\label{hyperbolic}
Let $G$ be a non-elementary hyperbolic group.  Then $E(G,S)=2$ for any finite
generating set $S$.  (That is, every presentation is statistically hyperbolic.)
\end{Thm}

\begin{proof}
Recall that $z$ is said to be (metrically) {\em between} $x$ and $y$ if
$d(x,z)+d(z,y)=d(x,y)$.  A set is between two other sets if there exists a triple of
points, one from each of the sets, satisfying the betweenness condition.

Choose any \(0<\rho<1\) and $x\in S_n$, and let $x'$ be an arbitrary point on $S_{\rho n}$
between $e$ and $x$.  We need to bound the number of $w\in S_n$ such that
$B_{2\delta}(x')$ is between $e$ and $w$.  But if $w'$ is a point in $S_{\rho n}$ between
$e$ and $w$, then $d(w',w)= n- \lfloor \rho n \rfloor$.  That means that the number of
such $w$ is overcounted by $|B_{2\delta}|\cdot |S_{n-\lfloor \rho n \rfloor }|$.

For every point $v$ of $S_n$ which is not of this kind, $d(x,v)\ge 2(n-\lfloor\rho
n\rfloor -\delta)\ge 2(n-\rho n - \delta)$ because the geodesics from the identity to $x$
and to $v$ have $2\delta$-diverged by time $\lfloor\rho n\rfloor$.  Thus,
\begin{equation}\tag{$\star$}\label{hyper}
\sum_{x,y\in S_n} d(x,y) \ge 2(n-\rho n -\delta)
\left( |S_n| - |B_{2\delta}|\cdot |S_{n-\lfloor \rho n \rfloor}| \right) \cdot |S_n|.
\end{equation}

Now we make use of the homogeneity.  Coornaert proved in \cite{coornaert} that for every
non-elementary hyperbolic group with fixed generating set,
there are bounded coefficients of exponential growth:
\begin{equation}\tag{$\dagger$}\label{coornaert}
\exists c_1,c_2>0, \ \omega>1 \qquad \hbox{s.t.} \qquad 
c_1 \omega^n \le \beta(n) \le c_2\omega^n
\qquad \forall n\in\N.
\end{equation}
 It follows from these inequalities that
$$\frac{|S_{n-\lfloor \rho n \rfloor}|}{|S_n|} 
= \frac{\beta(n- \lfloor \rho n  \rfloor )-\beta(n- \lfloor \rho n \rfloor  -1)}{\beta(n)-\beta(n-1)}\to 0$$ 
as $n\to \infty$, which together with \eqref{hyper} gives us
$$E(G,S) = \lim_{n\to\infty}\frac{\sum d(x,y)}{n|S_n|^2}  \ge 2(1-\rho).$$
Since $0<\rho<1$ was arbitrary, this means $E=2$.\end{proof}

To quickly clarify the necessity for the non-elementary hypothesis:  for $G=\Z$ and
any finite generating set, the spheres of large radius are divided into a positive part
and a negative part, each of uniformly bounded diameter.  Thus a pair of points has
bounded distance with probability $1/2$ and distance boundedly close to $2n$ with
probability $1/2$.  This gives $E(\Z,S)=1$ for all finite generating sets $S$.

\subsection{Some statistically hyperbolic groups and spaces}

Here we exhibit several examples of non-hyperbolic groups with statistically hyperbolic presentations.
We first consider groups that are direct products with a hyperbolic factor, and then use the results on 
products to consider Diestel-Leader graphs.

Let us say that a based space $(H,h_0)$ has {\em definite exponential growth} if 
the growth function $\beta(n)$ of balls of radius $k$ centered at $h_0$ in $H$ satisfies \eqref{coornaert}.
Given a sequence of finite sets $A_n$, we will say that  \textit{almost all} points of $A_n$ satisfy
a property $(P)$, or that the property has {\em full measure}, 
if the subset of elements satisfying $(P)$ has proportion tending to one as $n\to\infty$.

The subtlety in analyzing products is that the sphere of radius $n$ projects to not a sphere but to a ball
in each factor.  Thus we need estimates for distances when points are on spheres of different radii;
we can use definite exponential growth in one factor to get control on the difference in radius (so that most
of the projection is in an annulus $A_n$), and then use hyperbolicity to get the distance estimates.  We also need to know
that spheres in these annuli are {\em evenly covered} by which we mean that
there is a function $f_n:\N\to\N$ such that $\#\left(\pi^{-1}(h)\cap S_n^X\right)=f_n(|h|_H)$ for almost all $h\in A_n^H$.

In the following technical lemma, the reader should imagine that $H$ is a direct factor of $X$ and that $\pi:X\to H$
is coordinate projection.  Recall that a {\em semi-contraction} is a distance non-increasing map.

\begin{lemma}[Annulus lemma]
Let $(X,x_0)$ and $(H,h_0)$ be  based graphs, suppose $H$ is $\delta$-hyperbolic with definite exponential growth,
and fix any $0<\rho<1$.
Consider the annulus $A_n^H=B_n^H(h_0)\setminus B_{\rho n-1}^H(h_0)$ in $H$ and 
the sphere $S_n^X=S_n^X(x_0)$ in $X$. 
Let $\pi:(X,x_0)\to(H,h_0)$ be a semi-contraction, mapping almost all points of $S_n^X$ into $A_n^H$ such that
spheres in $A_n$ are evenly covered.
Then
$$\liminf_{n\to\infty}\frac{1}{|S_n|^2} \sum_{x,y\in S_n} \frac1n d(x,y) \ge 2\rho.$$
\end{lemma}

This is proved by showing that when $\rho n\le i,j\le n$, then the average distance between
a point in $S_i^H$ and a point in $S_j^H$ 
is bounded below by $i+j-2\delta-2\rho n - cn\omega^{-\rho n}$ for a constant $c$,
where $\omega$ is the growth rate of $H$, as in \eqref{coornaert}.

We will apply this lemma to products of the form $H\times K$ where $H$ is hyperbolic and $K$ grows strictly
slower---that is, $K$ has subexponential growth, or has a smaller exponential growth rate.
Let us say that a generating set for a product is {\em split} if every generator projects to the identity
in one of the factors.

\begin{prop}[Products with a dominated factor]
Suppose that $H$ is a non-elementary hyperbolic group, $K$ is finitely generated, and 
$S$ is a split finite generating set for $H\times K$ such that the growth
function of $H$ dominates the growth function of $K$ with generators projected to the factors from $S$.
Then $(H\times K,S)$ is a statistically hyperbolic presentation.
\end{prop}

\begin{proof}
Let $\pi$ be projection to the $H$ factor from $X=H\times K$ and note that 
$$S_k^X=\bigcup_{i=0}^kS_i^H\times S_{k-i}^K.$$
Thus one easily verifies the hypotheses of the annulus lemma.  Letting $\rho\to 1$ gives $E(X,S)=2$.
\end{proof}

Another class of statistically hyperbolic spaces is the {\em Diestel-Leader graphs}.  We describe them 
briefly here and refer the reader to \cite{woess} for a more thorough treatment and some relevant properties.
For $m,n\ge 2$, take an $(m+1)$-valent tree $T_1$ and an $(n+1)$-valent tree $T_2$.  
Choose ends and corresponding horofunctions
$f_1$ and $f_2$.  This gives height functions $h_1=f_1$ and $h_2=-f_2$ on the trees.  
We visualize $T_1$ as ``growing up''
from its end at height $-\infty$ and $T_2$ as ``hanging down'' from its end at $+\infty$.
The Diestel-Leader graph $DL(m,n)$ is defined to be the subspace of $T_1\times T_2$ on which $h_1=h_2$.
A height function $h$ is induced on this graph from the tree factors, since their height functions match.
Like Cayley graphs, Diestel-Leader graphs have vertex-transitive group actions by isometries, 
which guarantees that geometric invariants
of $DL(m,n)$ do not depend on the choice of basepoint.

These graphs are considered models for solvable geometry:  the structure described above is in precise analogy
with the geometry of Sol, which has hyperbolic plane factors in the place of trees.  
Eskin, Fisher, and Whyte \cite{efw}
exploit this analogy to completely classify Diestel-Leader graphs and spaces with Sol geometry up to quasi-isometry.
Furthermore, for $m\ge 2$, $DL(m,m)$ can be realized as Cayley graphs of solvable groups, namely the 
{\em lamplighter groups} $F\wr\Z$ where $F$ is a finite group of order $m$.

Denote the coordinate projections by $\pi_i:DL(m,n)\to T_i$ for $i=1,2$.
A geodesic $\gamma$ in $DL(m,n)$ is said to \textit{turn} if it switches from increasing in height to
decreasing in height or vice versa.  Geodesics in Diestel-Leader graphs have at most two turns.
The following lemma tells us that spheres of large radius in a Diestel-Leader graph are 
``concentrated in distant heights.''

\begin{lemma}[Concentration in height]
Let $X=DL(m,n)$ be a Diestel-Leader graph, $x_0\in X$ be a basepoint at height $0$, and $0<\rho<1$.
Denote by $S_k$ the sphere of radius $k$ in $ X$ centered at $x_0$.
For almost all $x\in S_k$, $\rho k\le |h(x)|\le k$.
If $m>n$, then for almost all $x\in S_k$, $\rho k\le h(x)\le k$.
\end{lemma}

\begin{proof}
Assume $m\ge n$ and
consider the problem of counting $k$-tuples  $(x_1,x_2,\ldots,x_k)$ of vertices of $X$
such that each pair $(x_i,x_{i+1})$ bounds an edge and the concatenation of these edges forms
a geodesic in $ X$.  If we start by choosing $x_1$ to
be immediately above $x_0$, then we have $m$ choices, since there are $m$ vertices of $T_1$ immediately above
$\pi_1(x_0)$ and only one vertex of $T_2$ above $\pi_2(x_0)$.  If $x_2$ is chosen immediately above $x_1$, then there are
$m^2$ choices for the pair $(x_1,x_2)$.  In general, there are $m^i$ ways to choose a tuple of vertices
$(x_1,\ldots,x_i)$ such that $x_{j+1}$ is immediately above $x_j$.  Suppose we now choose $x_{i+1}$ below $x_i$.
Then $\pi_1(x_{i+1})=\pi_1(x_{i-1})$, which means that we have lost one of our choices for a vertex in $T_1$.
This choice is replaced by the choice of a vertex $\pi_2(x_{i+1})$ immediately below $\pi_2(x_i)$ other than $\pi_2(x_{i-1})$.
There are $n-1$ such possibilities.  If we now continue choosing vertices to be decreasing in height, then
we continue replacing factors of $m^i$ with factors $k\le m$.  So turns in a geodesic reduce the number of choices
and geodesics continue in the same direction for as long as they can before turning.
If $m=n$, then the same argument applies if we begin choosing $x_1$ immediately below $x_0$, so geodesics
tend to end in heights which are distant in either the positive or the negative direction.
If $m>n$, then the above argument shows that geodesics in fact tend to end in high (positive) heights.
\end{proof}

\begin{theorem}[Diestel-Leader graphs]
For any $m,n\ge 2$, the Diestel-Leader graph $ X=DL(m,n)$ is statistically hyperbolic.
\end{theorem}

\begin{proof}
Let $S_k$ denote the sphere of radius $k$ centered at a point $x_0$ of height zero and let $\rho$ be fixed.
We will begin by considering the case where $m>n$.  $\pi_1: X\to T_1$
is a semi-contraction, since $\pi_1$ takes paths in $ X$ to paths in $T_1$ while preserving
their length.  By the previous lemma, a full-measure subset $U_k$ of
$S_k$ lies above height $\rho k$.  A similar argument can be used to show that the same is true
for $T_1$, and that a full-measure 
subset $V_k$ of the annulus $A_k=B_k \setminus B_{\rho k}$ centered at $\pi_1(x_0)$ lies above
height $\rho k$.  In fact, $V_k=\pi_1(U_k)$.  In order to apply the Annulus Lemma, we just need show
the even covering condition.  Suppose $y\in V_k$ with $d=d(y,\pi_1(x_0))$, and let $\gamma$
be a geodesic in $ X$ of length $k$ starting at $x_0$ and ending at a point of $\pi^{-1}(y)$.
Then $\pi_1\gamma|_{[0,d]}$ is a path from $\pi_1(x_0)$ to $y$.  Such a path may initially
decrease in height, and so choices are made in the $T_2$ coordinate.  But since $y$ is above height
zero, $\gamma$ must then come back up and any choices made in $T_2$ for the initial portion of $\gamma$
will have no effect on where $\gamma$ ends.  The only significant choices in $T_2$ for $\gamma$
occur after $\pi_1\gamma$ passes $y$ and turns around again.  This final downward portion of $\gamma$ has length $(k-d)/2$.
So $k$ and $d$ must have the same parity and the number of points in the preimage of $y$ is a function of $d$.
Thus we apply the annulus lemma to get $E( X)=2$.

If $m=n$, then above argument shows that the average distance between a pair of points in $S_k$ above height 0
is close to $2k$.  By symmetry, it follows that the average distance between a pair of points in $S_k$
below height 0 is also close to $2k$.  But a significant proportion of pairs $x,y\in S_k$ will have the property
that $h(x)>0$ and $h(y)<0$.  By the previous lemma we may again assume that $h(x)$ is close to $k$ and that
$h(y)$ is close to $-k$.  So the difference in heights, a lower bound on distance, is close to $2k$.
\end{proof}

When $m=n$, the Diestel-Leader graph $DL(m,m)$ can be realized as the Cayley graph of the lamplighter group
$\Z_m\wr\Z$ for a certain natural generating set (or, more generally, $F\wr \Z$ for any finite group $F$ of order $m$).  
These interesting solvable groups are not nilpotent and they are not finitely presented.

\begin{corollary}[Lamplighter groups]
The lamplighter groups $\Z_m\wr\Z$ have statistically hyperbolic presentations.
\end{corollary}

Finally, besides Euclidean space itself, the symmetric spaces of noncompact type also have $E=2$
(essentially because of the probability tending to zero that pairs of points lie in a common flat).
We know of no examples of groups of exponential growth with $E<2$, but because of the facts above
it would be natural to expect that groups of non-uniform exponential growth need not have $E=2$.

%%%%%%%%%%%%%%%%
\section{Reducing from free abelian groups to convex geometry}\label{limit-measure}

In the free abelian groups $\Z^d$, studying the large-scale metric geometry is greatly aided 
by the natural embedding in $\R^d$.  
It is known that the finite word metrics on $\Z^d$ are asymptotic to norms on $\R^d$
(originally due to Burago \cite{burago}, and shown by an elementary geometric argument in \cite{limit-measure}),
so that these norms can be thought of as {\em limit metrics} coming from group theory.  
Recall that any convex, centrally symmetric body in $\R^d$ induces a  {\em Minkowski norm},  
namely the norm for which that convex body is the unit ball.
If a generating set for $\Z^d$ is called $S$, let $| \W|$ denote the length of $\W\in\Z^d$ in the word metric, and
let 
$L$  be the boundary of the convex hull of $S$ in $\R^d$.  Then the Minkowski norm 
$\normL{\cdot}$ having $L$ as its unit sphere is the limit metric,  
in the sense that there is a constant $K$ depending on $S$ 
such that 
$$\|\W\|_L \le |\W | \le \|\W\|_L + K$$
for all $\W\in \Z^d$.
This limit shape 
$L$ describes the asymptotic shape of spheres in the sense that 
$\frac 1n S_n \to L$ (say as a Gromov-Hausdorff limit).

In an earlier paper, we proved counting results for spheres in word metrics on $\Z^d$, showing that 
counting measure on the discrete spheres $S_n$ converges to the {\em cone measure} $\muL$ on $L$,
as pictured in Figure~\ref{cone-measure}.   The case of that theorem that is useful for us here states that
$$\lim_{n\to\infty} \frac{1}{|S_n|^2} \sum_{\X,\Y\in S_n} \frac 1{n} d(\X,\Y) 
= \int_{L^2} \normL{\X-\Y} \ d\muL^2(\X,\Y).$$
(The original theorem addresses more general averaging problems.)
Thus it follows immediately that 
$$E(\Z^d,S)= \int_{L^2}  \normL{\X - \Y} \ d\muL^2(\X,\Y)$$
for all finite generating sets $S$.

\begin{figure}[ht]
\begin{center}
\begin{tikzpicture}
\begin{scope}[scale=1.3]
\draw[black, thick](0,-1)--(1,0)--(0,1)--(-1,0)--cycle;
\draw[orange,fill=orange!50](0,0)--(3/14,11/14)--(7/14,7/14)--cycle;
\draw[ultra thick,red](3/14,11/14)--(7/14,7/14);

\begin{scope}[xshift=3cm]
\draw[black, thick](-1,-1)--(0,-1)--(1,0)--(1,1)--(0,1)--(-1,0)--cycle;
\draw[blue!50!purple,fill=blue!25!purple!25](0,0)--(2/7,1)--(5/7,1)--cycle;
\draw[ultra thick,red](2/7,1)--(5/7,1);
\draw[blue!50!purple,fill=blue!25!purple!25](0,0)--(1/7,-6/7)--(4/7,-3/7)--cycle;
\draw[ultra thick,red](1/7,-6/7)--(4/7,-3/7);
\end{scope}

\begin{scope}[xshift=6cm, scale=1/2]
\draw[black, thick] (2,1)--(1,2)--(-1,2)--(-2,1)--(-2,-1)--(-1,-2)--(1,-2)--(2,-1)--cycle;
\draw[blue!20!green,fill=blue!10!green!40](-.3,2)--(.7,2)--(0,0)--cycle;
\draw[ultra thick,red](-.3,2)--(.7,2);
\draw[blue,fill=blue!40](4/3,5/3)--(2,1)--(0,0)--cycle;
\draw[ultra thick,red](4/3,5/3)--(2,1);
\end{scope}

\begin{scope}[xshift=9cm]
\draw[black, thick](0,0) circle (1);
\draw[brown, fill=brown!50](0,0)-- (100:1) arc (100:880/7:1)--cycle;
\draw[ultra thick,red] (100:1) arc (100:880/7:1);
\end{scope}
\end{scope}

\end{tikzpicture}
\end{center}
\caption{Six arcs are shown in red in this figure, each having cone measure $1/14$; in
other words, all of the colored regions have $1/14$ as much area as the convex body they
are in.  In the square and the hexagon, all sides have equal measure.  
On the other hand, for this octagon  generated by the {\em chess-knight} moves $\{(\pm 2,\pm 1),(\pm
1,\pm 2)\}$, the measure of its two types of sides (shown with green and blue) is in
the ratio $4:3$.  Cone measure is defined on any perimeter, and
in particular it is uniform on the circle.\label{cone-measure}}
\end{figure}
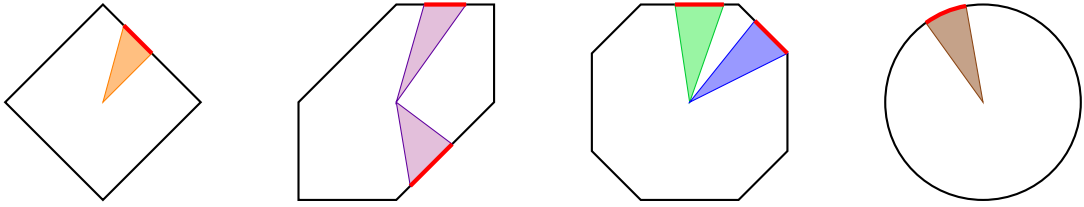

We can define the sprawl of
any perimeter $L$ by the right-hand side, which we can denote by
$E(L)$,  measuring average distance between points of $L$ as measured in its 
intrinsic geometry.  
We remark that $E(L)=E(TL)$ for any linear transformation $T:\R^d\to\R^d$,
since both the norm and the measure push forward under linear transformation.  That is,
$\normTL{T\X-T\Y} = \normL{\X-\Y}$, and $d\muTL(T\X)=d\muL(\X)$.

One immediate consequence of the reduction to convex geometry is that $E(\Z^d,S)$ is always greater
than $1/2$.

\begin{proposition} $E(L)>\frac 12$ for all perimeters $L$ in $\R^d$. \end{proposition}

\begin{proof}
Fix an arbitrary point $x\in L$, and denote by $Q$ the convex body of which $L$ is boundary.
The points of $L$ whose distance from $x$ is less than one are those contained
in $Q+x$, the translated copy of $Q$ centered at $x$.  Since $L+x$ contains $0$ and $Q$ is
convex, there is a hyperplane $P$ through $0$ which does not intersect the interior of $Q+x$.
So the interior of $Q+x$ is on one side of $P$, and by central symmetry, half of the cone measure
lies on each side of $P$.  Thus the average distance on $L$ from $x$ is $\ge (1/2)(1)$.  To obtain the strict
inequality, just note that the distance from $x$ to $-x$ is always $2$ and so a small neighborhood of $-x$ contributes
an amount near $2$ to the average.
\end{proof}

%%%%%%%%%%%%%%%%
\section{Sprawl in the plane}\label{plane}

From the work above, we have reduced the group calculation $E(\Z^d,S)$
to the convex geometry calculation $E(L)$.  In this section we study this convex geometry
in dimension $2$, by first introducing an algorithm for evaluating $E(L)$.
This algorithm can be given to a computer (which we did, producing a great deal of experimental 
evidence for the conjectures to follow) but can also be used to produce precise formulas, such as those
given below for the regular polygons.

\subsection{{\em Cutline} algorithm}\label{cutline}

To compute the sprawl of a polygon, we can average the expected distances between 
pairs of sides. 
Pick two sides $\sigma$ and $\tau$ of $L$ and parametrize each of them (say clockwise) by $[0,1]$;
then the distance in the $L$-norm from $\sigma(s)$ to $\tau(t)$ is piecewise linear.  Thus 
for an appropriate triangulation of the parameter space, average-distance is 
a linear function on each triangle.  
We outline here a method for triangulating, which we call the {\em cutline} algorithm for 
computing the sprawl of a polygon.  We note that the algorithm generalizes straightforwardly to
higher dimensions.

Fix $\sigma$ and $\tau$.
Find the sector of angles at which the 
sides ``see" each other---that is,  the interval of arguments obtained
by vectors from $\sigma(s)$ to $\tau(t)$---as in the first picture in Figure~\ref{cutlines}.  
Considering the same sector of angles viewed from the origin, as in the second picture,
mark the angles that point in vertex directions in this sector (shown as a dashed line).  

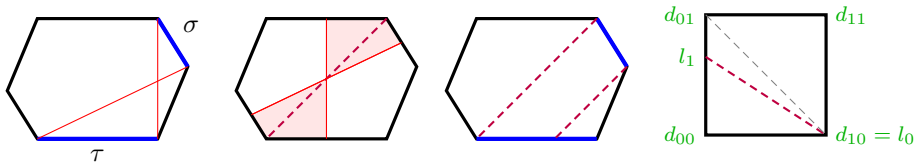
\begin{figure}[ht]

\begin{tikzpicture}[auto, scale=0.8]
\pgfsetbaseline{0cm}
\draw [very thick]
  (1,-1) -- (1.5,0.2) -- node [swap] {$\sigma$} (1,1) -- (-1,1) -- 
  (-1.5,-0.2) -- (-1,-1) -- node [swap] {$\tau$} (1,-1) -- cycle;
\draw [ultra thick,blue] (1.5,0.2) -- (1,1);
\draw [ultra thick,blue] (-1,-1) -- (1,-1);
\draw [red] (1.5,0.2) -- (-1,-1);
\draw [red] (1,1) -- (1,-1);
\end{tikzpicture}%
\quad%
\begin{tikzpicture}[scale=0.8]
\pgfsetbaseline{0cm}
\filldraw [draw=red,fill=red!10!white]
  (1.25,0.6) -- (-1.25,-0.6) -- (-1,-1) -- (0,-1) -- (0,1) -- (1,1) --
  (1.25,0.6) -- cycle;
\draw [very thick]
  (1,-1) -- (1.5,0.2) -- (1,1) -- (-1,1) -- (-1.5,-0.2) -- (-1,-1) -- cycle;
\draw [densely dashed, purple, thick] (1,1) -- (-1,-1);
\draw [red] (1.25,0.6) -- (-1.25,-0.6);
\draw [red] (0,1) -- (0,-1);
\end{tikzpicture}%
\quad%
\begin{tikzpicture}[scale=0.8]
\pgfsetbaseline{0cm}
\draw [very thick]
  (1,-1) -- (1.5,0.2) -- (1,1) -- (-1,1) -- (-1.5,-0.2) -- (-1,-1) -- cycle;
\draw [ultra thick,blue] (1.5,0.2) -- (1,1);
\draw [ultra thick,blue] (-1,-1) -- (1,-1);
\draw [densely dashed, purple, thick] (1,1) -- (-1,-1);
\draw [densely dashed, purple, thick] (1.5,0.2) -- (0.3,-1);
\end{tikzpicture}%
\quad%
\begin{tikzpicture}[scale=1.6]
\pgfsetbaseline{0.75cm}
\draw [very thick] (0,0) rectangle +(1,1);
\node at (0,1) [left,color=green!70!black] {\footnotesize $d_{01}$};
\node at (0,0.65) [left,color=green!70!black] {\footnotesize $l_1$};
\node at (0,0) [left,color=green!70!black] {\footnotesize $d_{00}$};
\node at (1,0) [right,color=green!70!black] {\footnotesize $d_{10}=l_0$};
\node at (1,1) [right,color=green!70!black] {\footnotesize $d_{11}$};
\draw [densely dashed, purple, thick] (0,0.65) -- (1,0);
\draw [densely dashed, black!50] (0,1) -- (1,0);
\end{tikzpicture}

\caption{A depiction of the algorithm for finding the average distance between 
sides $\sigma$ and $\tau$.}\label{cutlines}
\end{figure}

For each vertex direction $\theta$, consider the line $T\subset [0,1]\times [0,1]$ of times at which the vector between the 
sides points in the vertex direction; the corresponding chords form a trapezoid as in the third picture.
For each trapezoid, record the lengths of its bases, marked in the figure as $l_0$ and $l_1$.  
(In general, for direction $\theta_i$,
these are the largest and smallest values of $d(\sigma(s),\tau(t))$ for $(s,t)\in T_i$, and can be denoted $l_{i0}$ and $l_{i1}$.)
Let $d_{00},d_{01},d_{10},d_{11}$
be the four distances between an endpoint of $\sigma$ and an endpoint of $\tau$
(measured in the $L$-norm), with $d_{ij}=d(\sigma(i),\tau(j))$.

Next, consider the unit square formed by the parameters $[0,1]\times [0,1]$.  All the distances between points on 
the two chosen sides of the polygon can be recorded by a real-valued function on this square.  
To find the average distance between sides $\sigma$ and $\tau$, we only need to integrate that function over the 
square (using Lebesgue measure because the cone measure is proportional to arclength on each side; the proper
weights will be restored below).
Since the function is piecewise linear, it will suffice to know its values at the points of a triangulation that is fine enough
that the function is linear on each triangle.

For each vertex direction $\theta_i$, the corresponding times $T_i$ cut out a straight segment
across the square, which we will call a {\em cutline}.  
The values at the corners of the square are the $d_{ij}$ and the values at the endpoints of the cutlines are the $l_{ij}$.
If the cutlines do not triangulate the square,
add dummy cutlines as needed (between these same points, so requiring no further distance calculations) to complete a triangulation.  
One such dummy cutline is shown in the figure.

Now the average distance between a point on side $\sigma$ and a point on side $\tau$ can
be read off of this parameter square by just knowing the values at the vertices of the triangles:
for each triangle, average the values at its vertices, and then sum those averages over
all the triangles, weighted by the areas of the triangles.  Thus let $E_{ij}$ denote
the average distance between $\sigma_i$ and $\sigma_j$.
Let $w_i$ be the weight of the $i$th side in the cone measure:  
$w_i= \muL(\sigma_i)$.
Then, finally, the average distance between all pairs of points on the polygon
can be written as the weighted average:
$$E(L) = \frac{\sum_{i,j} w_iw_j E_{ij}}{\sum_{i,j} w_iw_j}.$$

%%%
\subsection{Values}

By applying the cutline algorithm, we find formulas for the sprawls of regular polygons.  
We note that the regular hexagon is equivalent by linear transformation to the 
hexagon with vertices $\pm (1,0), \pm (1,1), \pm (0,1)$, which is the limit set
for the generating set $S=\pm\{\E_1,\E_2,\E_1+\E_2\}$.  For regular polygons with 
at least $8$ sides, however, they are not exactly realized by word metrics on $\Z^2$.

\begin{proposition}\label{polygons}  Let $P_k$ be the regular $k$-gon
and let $S^1$ be the round unit circle.  Then 
$$\begin{array}{rll}

\smallskip

E(\Z^2, \pm \{ \E_1,\E_2 \})=E(P_4) &=& \frac 43; \\

\smallskip

E(\Z^2, \pm \{ \E_1,\E_2, \E_1+\E_2 \}) = E(P_6) &=& \frac{23}{18}<\frac 43;\\

\smallskip

E(P_8) &=&\displaystyle \frac {1+2\sqrt 2}3 < \frac{23}{18};\\

\smallskip

E(P_{x}) &=& \left\{ \begin{array}{ll}
\smallskip

\displaystyle\frac 4\pi \cdot 
\left( \displaystyle\frac{\pi/x}{\tan(\pi/x)} + 
\frac 13 \textstyle\frac{\pi}x \tan(\textstyle\frac{\pi}x)  \right), & x\in 4\N,\\
\displaystyle \frac 4\pi \cdot 
\left( \displaystyle\frac{\pi/x}{\sin(\pi/x)} - 
\frac 16 \textstyle\frac{\pi}x \sin(\textstyle\frac{\pi}x) \right), & x\in 4\N+2,
\end{array}
\right.\\

\smallskip

E(S^1) &=& \displaystyle \frac 4\pi.
\end{array}$$
\end{proposition}

Proposition~\ref{polygons}, shown in Figure~\ref{Z2numberline} below, 
shows of the nontrivial dependence of sprawl on the choice of generating set.
Since the word metrics of a group $G$ with respect to finite generating sets $S,S'$ 
are quasi-isometric, we see that sprawl is not a quasi-isometry invariant.

To prove the formula for regular polygons, one can set $a_j$ for the average 
distance from $\sigma_1$ to $\sigma_j$ and re-express that using the chordlengths
$\ell_i=d(\V_1,\V_i)$, by the cutline algorithm.  The $\ell_i$ themselves can then 
be written as trigonometric functions of $\pi/x$.  Trigonometric identities 
finish the proof, since $E(P_x)$ is the weighted average of the $a_j$.  

We note that the formulas for regular polygons each converge quickly to $4/\pi$, and track
close together.  Writing $E_{4\N}(x)$ for a function whose values agree with 
$E(P_x)$ when $x\in 4\N$, and likewise $E_{4\N+2}(x)$, we have:

%\begin{equation}\tag{$\searrow$}\label{approach}
$$E_{4\N}(x) - \frac{4}{\pi} \sim \frac{16\pi^3}{45 x^4}, \qquad 
E_{4\N+2}(x) - \frac{4}{\pi} \sim \frac{17\pi^3}{90 x^4}, \qquad
E_{4\N}(x)-E_{4\N+2}(x) \sim \frac{\pi^3}{6x^4}.$$
%\end{equation}

\begin{figure}[ht]
\begin{tikzpicture}[scale=8]

\begin{scope}[xshift=-14.1cm,yshift=-.4cm]
\begin{scope}[scale=12]
\draw [->](4/pi - .05,.0075)--(4/3 +.05,.0075);
\draw [<-](4/pi - .05,.0075)--(4/3 +.05,.0075);
    \draw [fill, fill=red!20!white, rounded corners=.1cm,opacity=.5]
      (4/pi,0) rectangle (4/3,.015);
\draw (1.25,0.005)--(1.25,0.01) node [above] {\scriptsize $1.25$};
\draw (1.35,0.005)--(1.35,0.01) node [above] {\scriptsize $1.35$};
\draw (4/pi,.015)--(4/pi,0);
\draw (4/pi -.01,-.01) -- (4/pi,0);
%\draw [red] (1129/882,.015)--(1129/882,0);
\draw (1.27326455338546,.015)--(1.27326455338546,0);
\draw (1.27382605804159,.015)--(1.27382605804159,0);
\draw (1.276142374915397,.015)--(1.276142374915397,0);
\draw (1.2745,-.01) -- (1.276142374915397,0);
\draw (1.286,-.01) -- (23/18,0);
\draw (23/18,.015)--(23/18,0);
\draw (4/3,.015)--(4/3,0) ;
\draw (4/3-.002,-.01) -- (4/3,0);
\draw (4/3+.01,-.01) to (4/3,0);
\draw [fill=gray!20, fill opacity=.5](4/pi -.01,-.01) circle (.005);
\node[draw,regular polygon,regular polygon sides=8,scale=3,fill=gray!20, fill opacity=.5] at (1.2745,-.01){};
\draw  [fill=gray!20, fill opacity=.5] (1.286,-.01)+(-.005,-.005) -- +(0,-.005) 
--+(.005,0) --+(.005,.005) -- +(0,.005) -- +(-.005,0) -- cycle;
\draw [fill=gray!20, fill opacity=.5] (4/3-.002,-.01)+(.005,0)--+(0,.005)--+(-.005,0)--+(0,-.005)--cycle;

\draw   [fill=gray!20, fill opacity=.5] (4/3+.01,-.01)+(-.005,-.005)
--+(.005,-.005)--+(.005,.005)--+(-.005,.005)--cycle;
\end{scope}
\end{scope}

\end{tikzpicture}
\caption{Range of sprawls known for $\Z^2$.}\label{Z2numberline}
\end{figure}
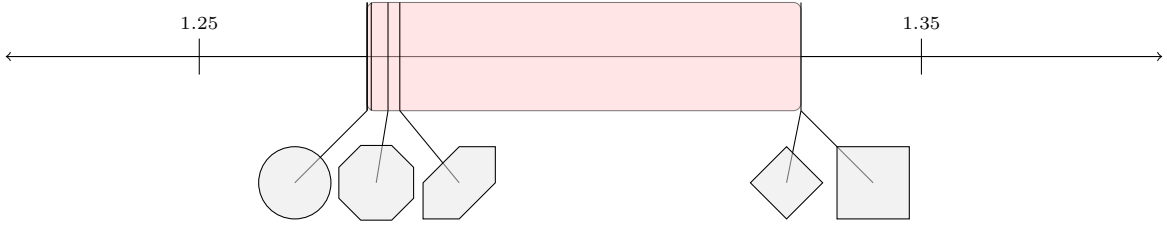

Using these proposition above 
and a rational approximation argument, we observe a range of sprawls that can be achieved in $\Z^2$.

\begin{corollary}\label{ratapprox}
 A dense subset of the interval $\left[ \frac 4 \pi , \frac 43 \right]$ is contained in the 
set $\{E(\Z^2,S) : \hbox{\rm gensets} ~ S\}$.
\end{corollary}

\begin{proof} 
There is a continuous path $L_t$ through the space of perimeters
that starts with the circle and ends with the square.  The sprawl passes through all 
values from $4/\pi$ to $4/3$ along the path.

Any such perimeter $L_t$ 
can be approximated arbitrarily closely by a rational polygon, which can be rescaled 
to an integer polygon without changing $E$.  The sprawl of a polygon is continuous
in the coordinates of its vertices, and $E$ of the approximants approaches $E$ of the original
body.  Finally, the set of integer vertices can be completed to a generating set without 
changing $E$, since the sprawl only depends on the extreme vertices.  
\end{proof}

\subsection{Hexagons}

Let $H_{x,y}$ be the hexagon with vertices
$\V_1 = (x,y)$, $\V_2 = (1,1)$, 
$\V_3 = (-1,1)$, where 
$x \geq 1$,  $y \geq 0$, and $x + y \leq 2$.
Thus $H_{1,0}$ is a square (realized as a degenerate hexagon)
and $H_{2,0}$ is a linear transform of the regular hexagon, giving
$$E(H_{1,0})=\frac 43, \qquad E(H_{2,0})=\frac{23}{18}.$$

\begin{lemma}[Parametrizing hexagons]
Every convex, centrally symmetric hexagon is equivalent by a linear transformation 
to some $H_{x,y}$.
\end{lemma}

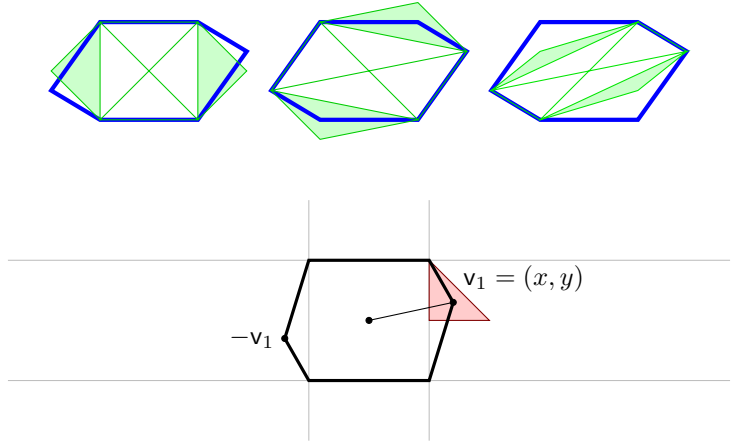
\begin{figure}[ht]
\begin{tikzpicture}[scale=0.65]
\pgfsetbaseline{0.65cm}
\begin{scope}
\fill [fill=green!20!white] (2,0) -- (1,1) -- (1,-1) -- cycle;
\fill [fill=green!20!white] (-2,0) -- (-1,-1) -- (-1,1) -- cycle;
\draw [ultra thick,blue]
  (-1,-1) -- (1,-1) -- (2,0.4) -- (1,1) -- (-1,1) -- (-2,-0.4) -- cycle;
\draw [thin,green!80!black]
  (-1,-1) -- (1,-1) -- (1,1) -- (-1,1) -- (-1,-1) -- (1,1) -- (2,0)
  -- (1,-1) -- (-1,1) -- (-2,0) -- cycle;
\end{scope}
\begin{scope}[xshift=4.5cm]
\fill [fill=green!20!white] (2,0.4) -- (-1,1) -- (1,1.4) -- cycle;
\fill [fill=green!20!white] (-2,-0.4) -- (1,-1) -- (-1,-1.4) -- cycle;
\draw [ultra thick,blue]
  (-1,-1) -- (1,-1) -- (2,0.4) -- (1,1) -- (-1,1) -- (-2,-0.4) -- cycle;
\draw [thin,green!80!black]
  (1,-1) -- (2,0.4) -- (-1,1) -- (-2,-0.4) -- (1,-1) -- (-1,1)
  -- (1,1.4) -- (2,0.4) -- (-2,-0.4) -- (-1,-1.4) -- cycle;
\end{scope}
\begin{scope}[xshift=9cm]
\fill [fill=green!20!white] (-2,-0.4) -- (1,1) -- (-1,0.4) -- cycle;
\fill [fill=green!20!white] (2,0.4) -- (-1,-1) -- (1,-0.4) -- cycle;
\draw [ultra thick,blue]
  (-1,-1) -- (1,-1) -- (2,0.4) -- (1,1) -- (-1,1) -- (-2,-0.4) -- cycle;
\draw [thin,green!90!black]
  (2,0.4) -- (1,1) -- (-2,-0.4) -- (-1,-1) -- (2,0.4) -- (-2,-0.4) -- 
  (-1,0.4) -- (1,1) -- (-1,-1) -- (1,-0.4) -- cycle;
\end{scope}
\end{tikzpicture}

\vspace{.3in}

\begin{tikzpicture}[scale=0.8]
\pgfsetbaseline{1.4cm}
  \draw[gray!50] (-6,1)
  --(6,1);
  \draw[gray!50] (-6,-1) 
  --(6,-1);
    \draw[gray!50] (-1,2)--(-1,-2);
  \draw[gray!50] (1,2) --(1,-2);
\filldraw [fill=red!20!white,draw=red!50!black]
  (1,0) -- (2,0) -- (1,1) -- cycle;
\draw [very thick]
  (1,-1)   -- (1.4,0.3) node [above right] {$\V_1 = (x,y)$}
  -- (1,1)   -- (-1,1)  -- (-1.4,-0.3) node [below=0.05pt, left = 0.05pt] {$-\V_1$}
  -- (-1,-1)  -- cycle;
  \filldraw (1.4,0.3) circle (0.05);
    \filldraw (-1.4,-0.3) circle (0.05);
        \filldraw (0,0) circle (0.05);
\draw [thin] (0,0) -- (1.4,0.3);
\end{tikzpicture}

\caption{Hexagon reduction.  On top we have shown the choices of which pair 
of sides to map to horizontal; the middle figure has $\V_1$ in the desired position.}
\label{hexagon-reduction}
\end{figure}

\begin{proof}
Take a hexagon with vertices $\V_1,\V_2,\V_3,-\V_1,-\V_2,-\V_3$.
We can always find a linear transformation sending $\V_2 \mapsto (1,1)$ and 
$\V_3\mapsto (-1,1)$.  This reduces the parameter space to 
$\{\V_1=(x,y) : x \geq 1, -1 \leq y \leq 1\}$.
Also, without loss of generality, we have $x + |y| \leq 2$; otherwise, change the choice
of $\V_2$,$\V_3$, as in Figure~\ref{hexagon-reduction}.
Finally, up to reflection in one of the coordinate axes, we can assume $y \geq 0$.
\end{proof}

Applying the algorithm sketched above, we can compute the side-pair 
averages, and obtain the following formula for a hexagon parametrized 
as above.

$$E(H_{x,y}) = \frac{x^2y^2 + xy^3 + 4x^3 + 7x^2 + 4x^2y - y^3 + 7xy - y^2 + 4x + 5y + 1}{3x^3 + 3x^2y + 6x^2 + 6xy + 3x + 3y}.$$

Thus we have reduced the task of bounding the sprawl of hexagons to a calculus exercise (which we omit):
verifying that in the domain defined by
$x\geq 1$, $y\geq 0$, and $x + y\leq 2$, this quantity
takes values between $23/18$ and $4/3$.

This establishes the following statement:
\begin{theorem}[Sprawls of hexagons and three-generator presentations]
$$\{ E(H) : ~\hbox{\rm hexagons}~ H\} = \left[ \frac{23}{18}, \frac 43 \right].$$
Thus,  $\frac{23}{18} \le E(\Z^2,S) \le \frac 43$ whenever $|S|\le 6$.
\end{theorem}

This provides evidence, taken together with the fast convergence for sprawls 
of regular polyhedra
towards $4/\pi$, for the following conjecture.

\begin{conjecture}[Sprawl Conjecture for $d=2$] The circle and the square are the extreme cases for all 
 perimeters in $\R^2$.
That is, 
$$\{E(L) : ~\hbox{\rm perimeters}~ L\subset\R^2 \} = 
\left[ \frac{4}{\pi}, \frac 43 \right].$$
\end{conjecture}

Further evidence is given in the next section, where the sphere and cube are shown to be sharp bounds
asymptotically as $d\to\infty$.

%
%We note that each such body $\Omega\subset\R^d$ can be parametrized as 
%$r:S^{d-1}\to \R$, where $r=r_\Omega$  records the distance
%from the origin in direction $\theta$.  Central symmetry is the condition $r(\theta)=r(-\theta)$.
%The convexity condition can be written down in terms of $r(\theta)$ as well. 
%(For instance, when $d=2$, convexity is the requirement that
%the unit tangent vector to the curve $\Omega$ varies monotonically as a circle-valued
%function, possibly with jumps if the shape has corners).

%We can write a nice formula for the sprawl if we introduce new functions 
%$R:S^{d-1}\times S^{d-1}\to \R $
%and $T:S^{d-1}\times S^{d-1} \to S^{d-1}$ 
%for the magnitude and argument of the differences in position:
%$$R(u,v)= \|r(v)\cdot v-r(u)\cdot u\| 
%\ , \qquad T(u,v)= \frac{ r(v)\cdot v-r(u)\cdot u }{ \|r(v)\cdot v-r(u)\cdot u\| }.$$
%
%Then we have 
%$$E(\Omega) = \frac{\displaystyle
%\iint_{(S^{d-1})^2}
%\frac{R(u,v)}{ r\circ T(u,v)}
%\cdot  r(u) ^d \cdot  r(v) ^d \,du \, dv}%
%{\displaystyle\left(\int_{S^{d-1}} r(u) ^d du\right)^2}.
%$$
%
%This might give an approach, along with the powerful inequalities in 
%convex geometry, to bounding the values of sprawl.

%%%%%%%%%%%%%%
\section{Sprawl in $d$ dimensions:  The not-so-flatness of $\Z^d$}\label{high-rank}

In higher dimensions, the computation of expected distance between two points
becomes quite intuitive for the sphere and the cube.  Suppose $d$ is very large.
For the round unit sphere $\Sphere_d \subset \R^d$, 
which induces the Euclidean metric as its Minkowski norm, take one point to 
be at the north pole without loss of generality.  Then concentration of measure
phenomena ensure that the second point is almost surely on the equator, so the distance
between them is nearly $\sqrt{1^2+1^2}=\sqrt 2$.  On the other hand, the cube induces
the sup metric.  In this case, the distance computation is performed by sampling the
random variable $|\X_i-\Y_i|$, which ranges between 0 and 2, a total of $d$ times.  
For very large $d$, we should expect this supremum to tend to $2$.  This reasoning
predicts that $E(\Sphere_d)\to\sqrt 2$ and $E(\Cube_d)\to 2$; the former can be approximated
and the latter can be exactly realized by a word metric.  What about 
the group $\Z^d$ with its standard generating set?  In dimension 2, this is isometric to the cube
 metric, but that is no longer true for $d>2$.  In dimension 3, the limit shape for the standard
word metric is an octahedron, and more generally 
 in dimension $d$ it is the join of $d$ copies of $S^0$, called 
 an {\em orthoplex} (or {\em cross-polytope}).
We will derive the answer below, finding that $E(\Z^d,\std)\to \frac 32$.

By way of interpretation, this says that a cubical generating set gives $\Z^d$ more and more 
hyperbolic-like geometry as $d$ gets large, while the standard word metric is bounded
uniformly away (see Figure~\ref{Zdnumberline}).
We are accustomed to describing the group $\Z^d$ as ``flat" because it is quasi-isometric
to Euclidean space.  However, using this statistic that gives a finer measure
of large-scale curvature, we see that the standard generators give more of a hyperbolic
character to the group, and that there exist generators for large $d$ which make the geometry
a good deal closer to hyperbolic than flat.

In the computations below, 
recall that for natural numbers $n$, the double factorial $n!!$ denotes the product of all the natural 
numbers up to $n$ that have the same parity:
$$    n!!=\prod_{i;~0\leq 2i<n}(n-2i).$$ 
Double factorials will occur in the calculations, but they can be re-expressed in two cases:
$$(2n)!! = 2^n\cdot n! \ ; \qquad (2n+1)!! = \frac{(2n+1)!}{2^n \cdot n!} .$$

To get rates of approach, we use an approximation for $n!$ that goes one term beyond 
Stirling's formula:
$$n!=\sqrt{2\pi n}\left({\textstyle \frac ne}\right)^n \left( 1 
+{\textstyle \frac{1}{12n}} + O({\textstyle \frac{1}{n^2}})\right).$$

%%%%%%%
\subsection{The sphere}
\begin{proposition}\label{spherecalc} The sphere induces the $\ell^2$ metric on $\R^d$.
The formula for the sprawl of the sphere is given in the following closed form:
$$ E(\Sphere_d) 
=   \frac{2^{d-1}}{\sqrt \pi} \frac{\Gamma(\frac 12 d)^2}{\Gamma(d-\frac {1}2)}.$$

Thus, $E(\Sphere_d)\to \sqrt 2$ as $d\to\infty$, with 
$$\sqrt 2 - E(\Sphere_d) \sim \frac{1}{8 d}.$$
\end{proposition}

\begin{proof}
Recall that, where $A_k$ denotes the surface area of $S^k$ (so that $A_1=2\pi$ and $A_2=4\pi$),
there is a recursive formula given by $A_k = \int_0^\pi A_{k-1} \sin^{k-1}(\theta) \ d\theta.$
The distance between two points on the sphere that subtend an angle $\theta$ at the 
origin is $\sqrt{2-2\cos\theta}$.
Then we find that the $A_k$ terms cancel out, giving 
$$E(\Sphere_d) = \frac{\int_0^\pi \sqrt{2-2\cos\theta}\cdot  \sin^{d-2}(\theta) 
\ d\theta}{\int_0^\pi  \sin^{d-2}(\theta) \ d\theta},$$
which can be computed explicitly.

Let
$$a_n = \int_0^\pi \sqrt{2 - 2 \stimes \cos \theta}\stimes \sin^n 
\theta\,d\stimes \theta,\qquad\qquad
b_n = \int_0^\pi \sin^n \theta\,d\stimes \theta,$$
so that $E(\Sphere_d)=a_{d-2}/b_{d-2}$.

Integrating by parts gives  $b_{n+2} = \frac{n+1}{n+2}\stimes b_n$, so since 
$b_0=\pi$ and $b_1=2$, we get $b_n=c_n \frac{(n-1)!!}{n!!}$, with 
$c_n=\pi$ if $n$ is even and $2$ if $n$ is odd.

Change of variables and integration by parts gives the recursion 
$a_{n+1} = \frac{2\stimes n + 2}{2\stimes n+3} a_n$.
Since $a_0 = 4$, this gives $a_n=4\frac{(2n)!!}{(2n+1)!!}$.

Combining and re-indexing wtih $d=n+2$, we get
$$E(\Sphere_d)=\frac{a_{d-2}}{b_{d-2}} = e_d \stimes \frac{(2d-4)!!}{(2d-3)!!} \stimes
\frac{(d-2)!!}{(d-3)!!}$$
where $e_d$ is $4/\pi$ if $d$ is even, and $2$ if $d$ is odd.
Re-expressing the double factorials completes the proof.  Note that the use of the gamma function
enables us to drop the dependence on parity of $d$ because 
$\Gamma(z)$ is an integer for whole numbers $z$ but has $\sqrt \pi$ in the denominator for half-integers
$z$.
\end{proof}

%%%%%%%%
\subsection{The cube}
\begin{proposition}\label{cubecalc}
The cube is the limit shape for $\Z^d$ with a nonstandard generating set 
$\{\pm \E_1  \cdots  \pm \E_d\}$, 
and it induces the $\ell^\infty$ metric on $\R^d$.
The formula for the sprawl of the cube is given in the following closed form:
$$E(\Cube_d) = \frac{2d+2}d - 
\left(\frac{2d+1}{2d^2}\right)
\left(\frac{4^d \ d!^2}{(2d)!}\right).$$

Thus, $E(\Cube_d)\to 2$ as $d\to\infty$, with 
$$2 - E(\Cube_d) \sim \frac{\sqrt \pi}{\sqrt d}.$$
\end{proposition}

\begin{proof}  Let $x_i$ and $y_i$ be independently distributed uniformly on the 
interval $I=[-1,1]$.
We will use these random variables to 
compute the sprawl for $\Cube_d$, which we identify with the $(d-1)$-complex
in $\R^d$ with vertices $(\pm 1,\ldots, \pm 1)$.
To fix notation:
$\Cube_1$ is a pair of points on the line and $\Cube_2$ is a square in the plane.  
$\Cube_d$ has $2d$ top-dimensional facets, each a copy of $I^{d-1}$.  Note that 
each facet is the locus of points satisfying $x_i=c$ for $c=\pm 1$.  It
has exactly one opposite face ($x_i=-c$), and all the others are adjacent since the defining equations
can be simultaneously satisfied.
For a point in $\R^d$ to be in $\Cube_d$, all coordinates must be in $I$, and 
at least one of its coordinates must be $\pm 1$.

We compute
$$P(|x_i-y_i|<r) = \frac{4r-r^2}4 \ ; \qquad   P(|1-y_i|)<r = \frac{r}2$$
by considering the uniform measure on the square $I^2$
and calculating the portion of the area between the lines $x-y=r$ and $y-x=r$ in the first case, 
and above the line $y=r$ in the second.  From this we get cumulative distribution functions
$$F_{\rm same}(r)=P(d_\infty(\X,\Y)<r : \X,\Y~\hbox{\rm on same face})
 = \frac{(4r-r^2)^{d-1}}{4^{d-1}};$$
$$F_{\rm adj}(r)=P(d_\infty(\X,\Y)<r : \X,\Y~\hbox{\rm on adjacent faces}) 
=  \frac{(4r-r^2)^{d-2}}{4^{d-2}} \left(\frac{r}{2} \right)^2  .$$

To find expectations, we integrate $\int_0^2 r F'(r)\ dr$.  

The $d$-cube has $2d$ faces, so if $\X$ is placed randomly, then the probability that 
$\Y$ is on the same face or on the opposite face is $1/2d$ in each case, while all of the other $2d-2$ faces
are in the adjacent case.  Recalling that the distance between any two points on opposite faces is $2$, 
we get 
$$E(\Cube_d) = \frac{1\cdot 2 + 1\cdot \int_0^2 r F'_{\rm same}(r)\ dr ) + (2d-2)\cdot \int_0^2 r F'_{\rm adj}(r) \ dr }{2d}.$$ 
From this and some algebraic manipulation we derive
$$E(\Cube_d)=\frac 1d + \frac{d-1}{4^{d-1}d}
\left[   2 \int_0^2 r^{d-1}(4-r)^{d-2}\ dr + (d-1) \int_0^2 r^d(4-r)^{d-2}\ dr 
+ (2-d)\int_0^2 r^{d+1}(4-r)^{d-3}\ dr   \right] .$$

Let's let $I_{m,n}=\int_0^2 r^m(4-r)^n\ dr$.  Integration by parts and some further manipulations will give recursive formulas, for instance
$$I_{n,n}= 2^2\cdot \frac{2n}{2n+1} I_{n-1,n-1},$$
which simplifies to $I_{n,n} = 2^{2n+1} \frac{(2n)!!}{(2n+1)!!}$ since $I_{0,0}=2$.

The $I_{n+1,n}$, $I_{n+2,n}$, and $I_{n+4,n}$ are derived similarly, from which we find
$$E(\Cube_d) = 2- \left( 2+\frac 1d \right)\frac{(2d-2)!!}{(2d-1)!!} + \frac 2d .$$
Re-expressing the double factorials completes the proof.
\end{proof}

%%%%%%%%%%%
\subsection{The orthoplex}

\begin{proposition}\label{orthocalc}
The orthoplex is the limit shape for $\Z^d$ with its standard generating set $\pm\{\E_i\}$, 
and it induces the $\ell^1$ metric on $\R^d$.
The formula for the sprawl of the orthoplex is given in the following closed form:
$$E(\Orth_d)=\frac{3d-2}{2d-1}.$$
Thus, $E(\Orth_d)\to \frac 32$ as $d\to\infty$, with 
$$\frac 32 - E(\Orth_d) \sim \frac{1}{4d}.$$
\end{proposition}

\begin{proof}
First note that by symmetry, the expectation of $\| \X-\Y \|_1$ is equal to 
$d$ times the expectation of $|x_1-y_1|$.  Thus 
$$E(\Orth_d)= d\int_{I^2} |x_1-y_1|\ d\mu(x_1)\ d\mu(y_1),$$
where $d\mu$ is the measure induced by $\mu$ on a single coordinate 
axis of $\R^d$.  That measure is given by 
$$d\mu(x_1) = \frac{(1-|x_1|)^{d-2}}{(d-2)!}\ dx_1,$$
as can be verified by considering how much volume the orthoplex has at height $x_1$.
We can renormalize to a probability measure by taking $\nu = \frac{(d-1)!}2\mu$, 
so that $\int_{{\Orth_d}^2} d\nu^2 = 1$.
Thus we are calculating 
$$E(\Orth_d) = d \int_{I^2} |x_1-y_1| \ d\nu^2 
= \frac{d(d-1)^2}4 \int_{I^2} |x-y| \cdot (1-|x|)^{d-2}(1-|y|)^{d-2}\ dx\ dy.$$
But again by symmetry, this is just 
$$2d(d-1)^2 \int_{x=0}^1 x(1-x)^{d-2}\int_{y=0}^x (1-y)^{d-2}dy \ dx.$$
Evaluating in $y$ and then performing light manipulation gives us
$$2d(d-1)\left[  \int_0^1 x(1-x)^{d-2}\ dx - \int_0^1 x(1-x)^{2d-3}\ dx \right] 
= 2d(d-1) \left[ \frac{1}{(d-1)d} - \frac{1}{(2d-2)(2d-1)} \right] = \frac{3d-2}{2d-1},$$
as desired.
\end{proof}

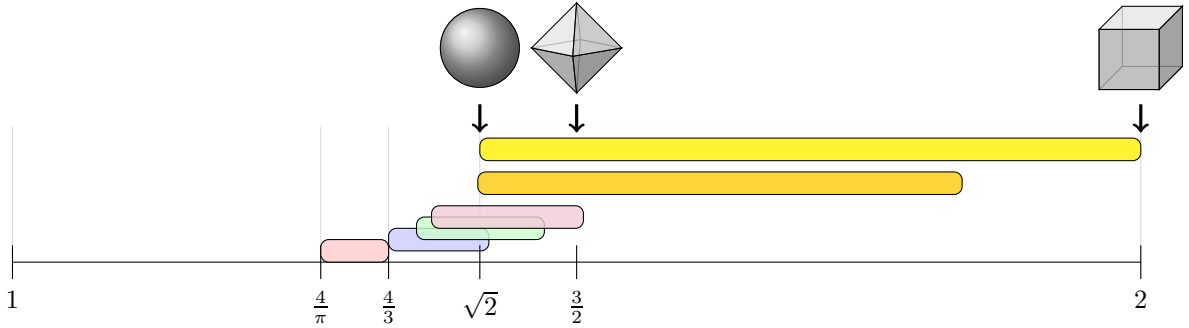
\begin{figure}[ht]
\begin{center}
\begin{tikzpicture}[scale=15]

  \foreach \x in {1,4/pi,4/3,1.41413562,2}
    \draw[gray!30] (\x,-0.015) -- (\x,0.12);
  \draw (1,0) -- (2,0);
  \draw (1,0) +(0,-0.015)
    node [below=1pt] {$1$}
    -- +(0,0.015);
  \draw (2,0) +(0,-0.015)
    node [below=1pt] {$2$}
    -- +(0,0.015);
  \draw (4/pi,0) +(0,-0.015)  
    node [below=1pt] {$\frac 4\pi$}
    -- +(0,0.015);
  \draw (4/3,0) +(0,-0.015) 
    node [below=1pt] {$\frac 43$}
    -- +(0,0.015);
  \begin{scope}
      \draw [fill, fill=red!20!white, rounded corners=3pt, fill opacity=.8]
      (4/pi,0) rectangle (4/3,0.02);
    \draw [fill, fill=blue!20!white, rounded corners=3pt, fill opacity=.8]
      (4/3,0.01) rectangle (64/45,0.03);
    \draw [fill, fill=green!20!white, rounded corners=3pt, fill opacity=.8]
      (64/15/pi,0.02) rectangle (103/70,0.04);
    \draw [fill, fill=purple!20!white, rounded corners=3pt, fill opacity=.8]
      (48/35,0.03) rectangle (2372/1575,0.05);
    \draw [fill, fill=yellow!80!red, rounded corners=3pt, fill opacity=.8]
      (1.412429119,0.06) rectangle (1.841645585,0.08);
    \draw [fill, fill=yellow, rounded corners=3pt, fill opacity=.8]
      (1.414213562,0.09) rectangle (2,0.11);
\draw [very thick, ->] (1.414213562,.14) -- (1.414213562,.115);
\draw [very thick, ->] (2,.14) -- (2,.115);

   \shadedraw [shading=ball,gray,ball color=gray!50]
      (1.414213562,0.19) circle (0.035cm);
   \draw       (1.414213562,0.19) circle (0.035cm);

\draw [very thick, ->] (3/2,.14) -- (3/2,.115);
  \draw (3/2,0) +(0,-0.015)
    node [below=1pt] {$\frac 32$}
    -- +(0,0.015);
\begin{scope}[xshift=3cm/2,yshift=0.19cm,scale=.04]
\draw (-1,0,0) -- (-.3,-.2,-1) -- (1,0,0);
\draw (0,-1,0) -- (-.3,-.2,-1) -- (0,1,0);
\draw [fill = gray!80, fill opacity = .8] (-1,0,0) -- (0,-1,0)
  -- (.3,.2,1) -- cycle;
\draw [fill = gray!20, fill opacity = .8] (-1,0,0) -- (0,1,0)
  -- (.3,.2,1) -- cycle;
\draw [fill = gray, fill opacity = .8] (1,0,0) -- (0,-1,0)
  -- (.3,.2,1) -- cycle;
\draw [fill = gray!50, fill opacity = .8] (1,0,0) -- (0,1,0)
  -- (.3,.2,1) -- cycle;
\end{scope}

\begin{scope}[xshift=2cm,yshift=0.19cm,scale=.08/3]
\draw (1,-1,-1) -- (-1,-1,-1);
\draw (-1,-1,1) -- (-1,-1,-1);
\draw (-1,1,-1) -- (-1,-1,-1);
\draw
  (-1,1,1) -- (1,1,1) -- (1,-1,1) -- (-1,-1,1) -- cycle;
\draw
  (1,1,1) -- (1,-1,1) -- (1,-1,-1) -- (1,1,-1) -- cycle;
\draw
  [fill = gray!20, fill opacity = .8]
  (-1,1,-1) -- (-1,1,1) -- (1,1,1) -- (1,1,-1) -- cycle;
\draw
  [fill = gray!60, fill opacity = .8]
  (-1,1,1) -- (1,1,1) -- (1,-1,1) -- (-1,-1,1) -- cycle;
\draw
  [fill = gray!90, fill opacity = .8]
  (1,1,1) -- (1,-1,1) -- (1,-1,-1) -- (1,1,-1) -- cycle;
\end{scope}
  \end{scope}
  \draw (1.41421356,0) +(0,-0.015)
    node [below=1pt] {$\sqrt{2}$}
    -- +(0,0.015);

\end{tikzpicture}
\end{center}
\caption{Ranges of sprawls: $\left[ E(\Sphere_d),E(\Cube_d)  \right]$ is shown for 
$d=2,3,4,5,100,\infty$.}\label{Zdnumberline}
\end{figure}

%%%%%%%%%%%
\subsection{The range of sprawls and the Mahler conjecture}

By rational approximation of convex bodies (as in the proof of Corollary~\ref{ratapprox}),
we find that a dense subset of the interval 
$\left[ E(\Sphere_d),E(\Cube_d)  \right]$ is contained in the set of values realized by groups, so
$$\left[ E(\Sphere_d),E(\Cube_d)  \right] \subseteq \overline{\{E(\Z^d,S) : \hbox{\rm gensets} ~S\}}.$$

We conclude by conjecturing that this is everything.

\begin{conjecture}[Sprawl Conjecture]\label{sprawl-range}     
The sphere and the cube are the extremes for the sprawl.  That is,
$$\{E(L) : ~\hbox{\rm perimeters}~ L\subset\R^d \} = 
\left[ E(\Sphere_d),E(\Cube_d)  \right].$$
\end{conjecture}

This conjecture would complete the description 
for free abelian groups of  the dependence of this curvature statistic on the generating set, showing
the values to be ``pinched" as in Figure~\ref{Zdnumberline}.

A similar conjecture could be formulated for the balls instead of the spheres:
consider the average distance statistic for convex, centrally symmetric $\Omega\subset\R^d$
defined by
$$\AD(\Omega):= \frac{\int_{\Omega^2} \|\X-\Y\|_\Omega \ d\vol^2}{(\vol\Omega)^2}.$$
Here, it is known (by the Brascamp-Lieb-Luttinger inequality \cite[Thm 1]{gluskin-milman})
that $\AD$ is minimized by (round) balls and ellipsoids, but the question of verifying that it 
is maximized by cubes is open.

Some evidence for the Sprawl Conjecture can be found in the high-dimensional asymptotics.
Because $E\le 2$ always, it is immediate that 
$$\lim_{d\to\infty}\sup\{E(L)\} = \lim_{d\to\infty} E(\Cube_d)=2.$$
Arias-de-Reyna, Ball, and Villa consider $\AD(\Omega)$ and 
prove that for almost all pairs of points in $\Omega\times\Omega$, the distance 
is greater than $\sqrt{2}(1-\epsilon)$~\cite[Thm 1]{ball}.  As they note,
the points in the ball become concentrated in its boundary as $d\to\infty$.
This shows that the $E(\Sphere_d)$ is a lower bound
for sprawl asymptotically, i.e., 
$$\lim_{d\to\infty}\inf\{E(L)\} = \lim_{d\to\infty} E(\Sphere_d)=\sqrt 2.$$

The Sprawl Conjecture resembles another well-studied problem in convex geometry.  
For a convex, centrally symmetric body $\Omega$, define its {\em polar body}
by 
$$\Omega^\circ := \{  \X\in \R^d : \X\cdot \Y \le 1 \quad \forall \Y\in\Omega  \}.$$
Thus for instance, the sphere is its own polar body in every dimension, 
$(\Omega^\circ)^\circ=\Omega$, and $(\Orth_d)^\circ=\Cube_d$.
The {\em Mahler volume} of $\Omega$ is defined to be 
$M(\Omega)=\vol(\Omega)\cdot\vol(\Omega^\circ)$.
Let us also say that for any set $A=-A$, we write 
$M(A)$ for the Mahler volume of the convex hull of $A$.
Then, just as for the sprawl, this is a statistic that is continuous in $\Omega$ and 
invariant under linear transformations; 
it has been described as measuring the ``roundness" of the convex body.  
Mahler conjectured in 1939 that the extremes in every dimension were realized by the
sphere and the cube.   Santal\'o proved in 1949 that the spheres did indeed realize
the upper bound on Mahler volume, but the lower bound is still an open problem,
despite some interesting recent progress by Kuperberg and others.  

Above, we have staked out the point of view 
that, like the Mahler volume and other affine isoperimetric invariants, 
sprawl is measuring a quality of roundness versus pointiness of the shape $L$.
Inspecting the estimates for sprawls of regular polygons derived after Proposition~\ref{polygons}
 shows something surprising:
there is no point after which sprawl decreases monotonically as the number of sides in the polygon 
increases.  
Thus, regular polygons with $4k-2$ sides are a bit ``rounder" than regular polygons with $4k$ sides 
(for all $k\ge 4$), even
though they have fewer sides.
On the other hand, as measured by Mahler volume the roundness of regular polygons increases monotonically
in the number of sides.  

Finally, we note that the average distance between two points on the round sphere
is precisely equal to the constant $\gamma_n$ that Kuperberg uses to 
state the inequality in his \cite[Corollary 1.6]{kuperberg}, where it is described 
as ``a monotonic factor that begins at $4/\pi$ and converges to $\sqrt 2$."  
Recognizing the geometric meaning of this 
constant allows his result  to be rephrased  as
$$M(\Omega) \ge   (\textstyle \frac{\pi}{4})^d \cdot E(\Sphere_d) \cdot M(\Cube_d).$$
The fact that this general inequality for Mahler volume should 
be so simply stated involving the sprawl is, we hope, intriguing.

%As a final remark, let $\Omega_p$ denote the unit ball in the $\ell^p$ norm on $\R^d$
%for $p\in [1,\infty]$. 
%These bodies vary continuously 
%from the orthoplex ($p=1$) to the cube ($p=\infty$) in each dimension.
%We note that Mahler volume is evidently invariant under polar dualization and thus 
%$M(\Omega_p)=M(\Omega_q)$ where $\frac 1p +\frac 1q=1$, while our 
%results above show that sprawl lacks this symmetry.
%For $M$, the conjectured extremes are at $p=2$ and $p=1,p=\infty$.  
%For $E$, they are at $p=2$ and $p=\infty$ with $p=1$ provably between.
%Indeed, crude
%experimental evidence supports the hypothesis that $E(\Omega_p)$ 
%is monotone decreasing for $p\in [1,2)$ and increasing for $p\in(2,\infty]$, thus
%taking its unique minimum at the round sphere ($p=2$).  

\bibliography{sprawl}
\bibliographystyle{siam}

\end{document}